\documentclass[11pt]{amsart}

\usepackage{amssymb,amsmath,color,hyperref}
\usepackage[mathscr]{eucal}

\theoremstyle{plain}
\newtheorem{thm}{Theorem}[section]
\newtheorem{theorem}[thm]{Theorem}

\newtheorem{prop}[thm]{Proposition}

\theoremstyle{definition}

\newtheorem{ex}[thm]{Example}

\theoremstyle{remark}

\newtheorem*{remark*}{Remark}

\numberwithin{equation}{section}

        \newcommand{\field}[1]{{\mathbb{#1}}}
        \newcommand{\NN}{\field{N}}
        \newcommand{\ZZ}{\field{Z}}
        
        \newcommand{\RR}{\field{R}}
        \newcommand{\CC}{\field{C}}

\allowdisplaybreaks

\begin{document}

\title[Semiclassical trace formula]{Semiclassical trace formula for the Bochner-Schr\"odinger operator}

\author[Y. A. Kordyukov]{Yuri A. Kordyukov}
\address{Institute of Mathematics, Ufa Federal Research Centre, Russian Academy of Sciences, 112~Chernyshevsky str., 450008 Ufa, Russia} \email{yurikor@matem.anrb.ru}



\begin{abstract}
We study the semiclassical Bochner-Schr\"odinger operator $H_{p}=\frac{1}{p^2}\Delta^{L^p\otimes E}+V$ on tensor powers $L^p$ of a Hermitian line bundle $L$ twisted by a Hermitian vector bundle $E$ on a Riemannian manifold of bounded geometry. For any function $\varphi\in C^\infty_c(\mathbb R)$, we consider the bounded linear operator $\varphi(H_p)$ in $L^2(X,L^p\otimes E)$ defined by the spectral theorem. We prove that its smooth Schwartz kernel on the diagonal admits a complete asymptotic expansion in powers of $p^{-1}$ in the semiclassical limit $p\to \infty$. In particular, when the manifold is compact, we get a complete asymptotic expansion for the trace of $\varphi(H_p)$. 
\end{abstract}

\date{}

 \maketitle
\section{Introduction}
Let $(X,g)$ be a complete Riemannian manifold of dimension $d$, $(L,h^L)$ a Hermitian line bundle on $X$ with a Hermitian connection $\nabla^L$ and $(E,h^E)$ a Hermitian vector bundle on $X$ with a Hermitian connection $\nabla^E$. We suppose that $(X, g)$ is a Riemannian manifold of bounded geometry and the bundles $L$ and $E$ have bounded geometry. This means that the curvatures $R^{TX}$, $R^L$ and $R^E$ of the Levi--Civita connection $\nabla^{TX}$, connections $\nabla^L$ and $\nabla^E$, respectively, and their covariant derivatives of any order are uniformly bounded on $X$ in the norm induced by $g$, $h^L$ and $h^E$, and the injectivity radius $r_X$ of $(X, g)$ is positive. 

For any $p\in \NN$, let $L^p:=L^{\otimes p}$ be the $p$th tensor power of $L$ and let
\[
\nabla^{L^p\otimes E}: {C}^\infty(X,L^p\otimes E)\to
{C}^\infty(X, T^*X \otimes L^p\otimes E)
\] 
be the Hermitian connection on $L^p\otimes E$ induced by $\nabla^{L}$ and $\nabla^E$. The Bochner Laplacian $\Delta^{L^p\otimes E}$ acts on $C^\infty(X,L^p\otimes E)$ by the formula
\[
\Delta^{L^p\otimes E}=\big(\nabla^{L^p\otimes E}\big)^{\!*}\,
\nabla^{L^p\otimes E},
\]
where $\big(\nabla^{L^p\otimes E}\big)^{\!*}: {C}^\infty(X,T^*X\otimes L^p\otimes E)\to
{C}^\infty(X,L^p\otimes E)$ is the formal adjoint of  $\nabla^{L^p\otimes E}$. Let $V\in C^\infty(X,\operatorname{End}(E))$ be a self-adjoint endomorphism of $E$. 
We consider a second order differential operator $H_p$ acting on $C^\infty(X,L^p\otimes E)$ by
\begin{equation}\label{e:Hp}
H_{p}=\frac{1}{p^2}\Delta^{L^p\otimes E}+V. 
\end{equation} 

\begin{ex}\label{ex:op}
Assume that the Hermitian line bundle $(L,h^L)$ is trivial: $$L=X\times \mathbb C,\quad |(x,z)|_{h^{L}}=|z|^2,\quad (x,z)\in X\times \mathbb C,
$$ the Hermitian connection $\nabla^L$ has the form $\nabla^L=d-i \mathbf A$ with a real-valued 1-form $\mathbf A$ (the magnetic potential), the bundle $(E,h^E)$ is the trivial Hermitian line bundle, and the connection $\nabla^E$ is trivial: $\nabla^E=d$. Then the operator $H_p$ coincides with the semiclassical magnetic Schr\"odinger operator 
\[
H_p=(ih d+\mathbf A)^*(ih d+\mathbf A)+V, \quad h=\frac{1}{p},\quad p\in \NN. 
\]
The magnetic field $\mathbf B=d\mathbf A$ is related with the curvature $R^L$ of $\nabla^L$ by the formula
\[
\mathbf B=iR^L. 
\]

If $X$ is the Euclidean space $\RR^{d}$ with coordinates $Z=(Z_1,\ldots,Z_{d})$, we can write the 1-form $\bf A$ as
\[
{\bf A}= \sum_{j=1}^{d}A_j(Z)\,dZ_j,
\]
the matrix of the Riemannian metric $g$ as $g(Z)=(g_{j\ell}(Z))_{1\leq j,\ell\leq d}$
and its inverse as $g(Z)^{-1}=(g^{j\ell}(Z))_{1\leq j,\ell\leq d}$.
Denote $|g(Z)|=\det(g(Z))$. Then $\bf B$ is given by 
\[
{\bf B}=\sum_{j<k}B_{jk}\,dZ_j\wedge dZ_k, \quad
B_{jk}=\frac{\partial A_k}{\partial Z_j}-\frac{\partial
A_j}{\partial Z_k}.
\]
Moreover, the operator $H_p$ has the form
\[
H_p=\frac{1}{\sqrt{|g|}}\sum_{1\leq j,\ell\leq d}\left(i h\frac{\partial}{\partial Z_j}+A_j\right) \left[\sqrt{|g|} g^{j\ell} \left(i h\frac{\partial}{\partial Z_\ell}+A_\ell\right)\right]+V.
\]
Our assumptions hold, if for any $\alpha \in \ZZ^{d}_+$ and $1\leq j,\ell\leq d$, we have 
\[
\sup_{Z\in \RR^{d}}|\partial^\alpha g_{j\ell}(Z)|<\infty, \quad \sup_{Z\in \RR^{d}}|\partial^\alpha B_{j\ell}(Z)|<\infty, 
\]
\[
\sup_{Z\in \RR^{d}}|\partial^\alpha V(Z)|<\infty, 
\]
and the matrix $(g_{j\ell}(Z))$ is positive definite uniformly on $Z\in \RR^{d}$. 

As shown in \cite[Section 4.4]{HP10}, using Agmon exponential decay estimates, the case of general $\mathbf B$ and $V$ can be reduced to the case of $\mathbf B$ and $V$, satisfying the above conditions. 
\end{ex}

Since $(X,g)$ is complete, the operator $H_p$ is essentially self-adjoint  in the Hilbert space $L^2(X,L^p\otimes E)$ with initial domain  $C^\infty_c(X,L^p\otimes E)$, see \cite[Theorem 2.4]{ko-ma-ma}. We still denote by $H_p$ its unique self-adjoint extension. For any function $\varphi\in C^\infty_c(\mathbb R)$, one can define by the spectral theorem a bounded linear operator $\varphi(H_p)$ in $L^2(X,L^p\otimes E)$. This operator has smooth Schwartz kernel $$K_{\varphi(H_{p})}\in {C}^{\infty}(X\times X, \pi_1^*(L^p\otimes E)\otimes \pi_2^*(L^p\otimes E)^*)$$ with respect to the Riemannian volume form $dv_X$.  (Here $\pi_1,\pi_2 : X\times X\to X$ are the natural projections.)  

The main result of the paper is the following theorem. 

\begin{thm}\label{c:main} 
The following asymptotic expansion holds true uniformly in $x_0\in X$:
\[
K_{\varphi(H_{p})} (x_0,x_0)\sim p^{d}\sum_{r=0}^{\infty}f_{r}(x_0)p^{-r},\quad p\to \infty, 
\]
where $f_{r},$ $r=0,1,\ldots$, is a smooth section of the vector bundle $\operatorname{End}(E)$ on $X$. The leading coefficient of this expansion is given by
\begin{equation}\label{e:f0-d}
f_{0}(x_0)=\frac{1}{(2\pi)^{d}} \int_{\RR^{d}} \varphi(|\xi|^2 +V(x_0))d\xi.
\end{equation}
The next coefficients have the form
\begin{equation}\label{e:fr2-form}
f_r(x_0)= \frac{1}{(2\pi)^{d}} \sum_{\ell=1}^{N} \int_{\RR^{d}} P_{r,\ell,x_0}(\xi) \varphi^{(\ell-1)}(|\xi|^2+V(x_0))d\xi,\quad r\geq 1,
\end{equation}
{where $P_{r, \ell,x_0}(\xi)$ is a universal polynomial depending on a finite number of derivatives of $V$ and the curvatures $R^{TX}$, $R^L$ and $R^E$ in $x_0$.} 
\end{thm}

As an immediate consequence of Theorem \ref{c:main}, we get the following semiclassical trace formula for the operator $H_p$ in the case when $X$ is compact. 

\begin{thm}\label{t:trace} Assume that $X$ is compact. There exists a sequence of distributions $f_r \in \mathcal D^\prime(\mathbb R), r=0,1,\ldots$, such that for any $\varphi\in C^\infty_c(\mathbb R)$, we have an asymptotic expansion
\begin{equation}\label{e:trace}
\operatorname{tr} \varphi(H_{p})\sim p^{d}\sum_{r=0}^{\infty}f_{r}(\varphi) p^{-r}, \quad p \to \infty.
\end{equation}
\end{thm}

In the Euclidean space $\mathbb R^d$ equipped with the canonical Euclidean metric (see Example~\ref{ex:op}), the semiclassical trace formula for the magnetic Schr\"odinger operator was established by Helffer and Robert \cite{HR83}, using  pseudodifferential Weyl calculus. Here the expression for the coefficients was given in terms of the magnetic potential $\mathbf A$. More precise information on the coefficients is given in \cite{HR90}.  However, by gauge invariance, the trace and its expansion should depend only on $\mathbf B$. A gauge invariant expression for the coefficients in terms of  $\mathbf B$ was proved by Helffer and Purice \cite{HP10}, using the magnetic pseudodifferential calculus developed in \cite{IMP07,MP04,MPR07} (see also references therein). In \cite{CdV12}, Colin de Verdi\`ere suggested a simplified version of the expansion replacing the (non-unique) coefficients by functions which are uniquely defined and are given by universal $O(d)$-invariant polynomials of the Taylor expansions of $\mathbf B$ and $V$ at the point $x$.

Our main results, Theorems~\ref{c:main} and \ref{t:trace}, extend the results of \cite{HP10}, but their proofs are quite different. We don't use any pseudodifferential techniques (to our knowledge, there is no constructions of magnetic pseudodifferential calculus on (noncompact) manifolds). Instead, we use a modification of the approach developed for the asymptotic analysis of generalized Bergman kernels by Dai-Liu-Ma and Ma-Marinescu \cite{dai-liu-ma,ma-ma:book,ma-ma08}, which was inspired by Bismut-Lebeau localization technique \cite{BL}, in combination with the functional calculus based on the Dynkin-Helffer-Sj\"{o}strand formula. Such a strategy has been applied in the close settings in \cite{Savale17}, \cite{Marinescu-Savale} and \cite{Bochner-trace}. Note that we do not require the non-degeneracy assumption on the curvature $R^L$ as in \cite{dai-liu-ma,ma-ma:book,ma-ma08}. Recall that, in \cite{ma-ma08,Bochner-trace}, the Bochner-Schr\"odinger operator is considered in an asymptotic regime, which is different from \eqref{e:Hp}, namely, the operator has the form
\[
\mathfrak H_{p}=\frac{1}{p}\Delta^{L^p\otimes E}+V. 
\] 
The key observation of our paper is that the operator $H_p$ in the form \eqref{e:Hp} can be treated in a similar way as in \cite{ma-ma08,Bochner-trace}, avoiding any use of pseudodifferential techniques.  

As in \cite{ma-ma08}, our main technical result, Theorem~\ref{t:main} below, is stronger than Theorem~\ref{c:main}. It states the existence of a complete asymptotic expansion for the Schwartz kernel $K_{\varphi(H_{p})}$ as $p\to \infty$ in some neighborhood of the diagonal (dependent of $p$) called the near off-diagonal asymptotic expansion. We believe that, using weighted estimates as in \cite{Bochner-trace}, one could establish the full off-diagonal expansion for $K_{\varphi(H_{p})}$,  that is, a complete asymptotic expansion in a fixed neighborhood of the diagonal (independent of $p$), but we don't discuss this issue here. 
Let us state the result in more detail. 

First, we introduce normal coordinates near an arbitrary point $x_0\in X$. 
We denote by $B^{X}(x_0,r)$ and $B^{T_{x_0}X}(0,r)$ the open balls in $X$ and $T_{x_0}X$ with center $x_0$ and radius $r$, respectively. We identify $B^{T_{x_0}X}(0,r_X)$ with $B^{X}(x_0,r_X)$ via the exponential map $\exp^X_{x_0}: T_{x_0}X \to X$. Furthermore, we choose trivializations of the bundles $L$ and $E$ over $B^{X}(x_0,r_X)$,   identifying their fibers $L_Z$ and $E_Z$ at $Z\in B^{T_{x_0}X}(0,r_X)\cong B^{X}(x_0,r_X)$ with the spaces  $L_{x_0}$ and $E_{x_0}$ by parallel transport with respect to the connections $\nabla^L$ and $\nabla^E$ along the curve $\gamma_Z : [0,1]\ni u \to \exp^X_{x_0}(uZ)$.

Consider the fiberwise product $TX\times_X TX=\{(Z,Z^\prime)\in T_{x_0}X\times T_{x_0}X : x_0\in X\}$. Let $\pi : TX\times_X TX\to X$ be the natural projection given by  $\pi(Z,Z^\prime)=x_0$. The Schwartz kernel $K_{\varphi(H_{p})}$ induces a smooth section $K_{\varphi(H_{p}),x_0}(Z,Z^\prime)$ of the vector bundle $\pi^*(\operatorname{End}(E))$ on $TX\times_X TX$ defined for all $x_0\in X$ and $Z,Z^\prime\in T_{x_0}X$ with $|Z|, |Z^\prime|<r_X$:
\begin{equation}\label{e:Kphix0}
K_{\varphi(H_p),x_0}(Z,Z^\prime)=K_{\varphi(H_p)}(\exp^X_{x_0}(Z),\exp^X_{x_0}(Z^\prime)). 
\end{equation}

Let $dv_{X,x_0}$ denote the Riemannian volume form of the Euclidean space $(T_{x_0}X, g_{x_0})$. We define a smooth function $\kappa_{x_0}$ on $B^{T_{x_0}X}(0,r_X)\cong B^{X}(x_0,r_X)$ by the equation
\begin{equation}\label{e:def-kappa}
dv_{X}(Z)=\kappa_{x_0}(Z)dv_{X,x_0}(Z), \quad Z\in B^{T_{x_0}X}(0,r_X). 
\end{equation}

\begin{theorem}\label{t:main}
There exists a sequence $F_{r,x_0}(Z,Z^\prime),$ $r=0,1,\ldots$, of smooth sections of the vector bundle $\pi^*(\operatorname{End}(E))$ on $TX\times_X TX$ such that, for any $j,m\in \mathbb N$ and $\sigma>0$, there exists $C>0$ such that, for any $p\geq 1$, $x_0\in X$ and $Z,Z^\prime\in T_{x_0}X$, $|Z|, |Z^\prime|<\sigma/p$, 
\begin{multline}\label{e:main-exp}
\sup_{|\alpha|+|\alpha^\prime|\leq m}\Bigg|\frac{\partial^{|\alpha|+|\alpha^\prime|}}{\partial Z^\alpha\partial Z^{\prime\alpha^\prime}}\Bigg(p^{-d}K_{\varphi(H_{p}),x_0}(Z,Z^\prime)\\
-\sum_{r=0}^jF_{r,x_0}(pZ, pZ^\prime)\kappa_{x_0}^{-\frac 12}(Z)\kappa_{x_0}^{-\frac 12}(Z^\prime)p^{-r}\Bigg) \Bigg|\leq Cp^{-j+m-1}.
\end{multline}
\end{theorem}

Putting $Z=Z^\prime=0$ in \eqref{e:main-exp}, we get the asymptotic expansion of Theorem \ref{c:main} with the coefficients  $f_{r}(x_0)$ given by
\[
f_{r}(x_0)=F_{r,x_0}(0,0), \quad r=0,1,\ldots.
\]

The proof of Theorem~\ref{t:main} follows the scheme of the proof of Theorem 1.19 in \cite{ma-ma08} on the near off-diagonal expansion of Bergman kernels modified as in \cite{Bochner-trace} to deal with general functions of the operator.  The main difference from \cite{ma-ma08,Bochner-trace} is that we use a different rescaling that results in different model operators, asymptotic expansions and estimates.  

So the first step (Section~\ref{local}) is a choice of appropriate local coordinates and trivialization of the line bundle and a localization of the problem near each point of the manifold that results in a family of second order differential operators in the fibers of the tangent bundle (model operators). Then, in Section~\ref{scale}, we rescale each operator of this family and derive its formal asymptotic expansion as $p\to \infty$. Finally, in Sections~\ref{norm} and \ref{proof}, we use the Dynkin-Helffer-Sj\"{o}strand formula, the formal power series technique, Sobolev norm estimates and Sobolev embedding theorem to derive a complete asymptotic expansions of Theorem~\ref{t:main} and compute their coefficients.
As an illustration, in Section \ref{s:ex-computation}, we consider the semiclassical magnetic Schr\"odinger operator in the Euclidean space (see Example~\ref{gauge}) and derive explicit formulas for the coefficients $f_1$ and $f_2$ stated in \cite[Theorem 1]{CdV12}.


 \section{Localization of the problem}\label{local}
In this section, we localize the problem near each point of the manifold, slightly modifying the constructions of \cite[Sections 1.1 and 1.2]{ma-ma08}. 

First, we will use normal coordinates near an arbitrary point $x_0\in X$ and trivializations of vector bundles $L$ and $E$ defined in Introduction. 
Consider the trivial bundles $L_0$ and $E_0$  on $T_{x_0}X$ with fibers $L_{x_0}$ and $E_{x_0}$, respectively. The above identifications induce the Riemannian metric $g$ on $B^{T_{x_0}X}(0,\varepsilon)$
as the connections $\nabla^L$ and $\nabla^E$ and the Hermitian metrics $h^L$ and $h^E$ on the restrictions of $L_0$ and $E_0$ to $B^{T_{x_0}X}(0,\varepsilon)$.

Now we fix $\varepsilon \in (0,r_X)$ and we extend these geometric objects from $B^{T_{x_0}X}(0,\varepsilon)$ to $T_{x_0}X$ so that $(T_{x_0}X, g^{(x_0)})$ is a manifold of bounded geometry, $L_0$ and $E_0$ have bounded geometry and $V^{{(x_0)}}\in C^\infty_b(T_{x_0}X,\operatorname{End}(E_0))$ be a self-adjoint endomorphism of $E_0$.  For instance, we can take a smooth even function $\rho : \mathbb R\to [0,1]$ supported in $(-r_X,r_X)$ such that $\rho(v)=1$ if $|v|<\varepsilon$ and introduce the map $\varphi : T_{x_0}X\to T_{x_0}X$ defined by $\varphi (Z)=\rho(|Z|_{g^{(x_0)}})Z$. Then we introduce the Riemannian metric $g^{(x_0)}_Z=g_{\varphi(Z)}, Z\in T_{x_0}X$, a Hermitian connection $\nabla^{L_0}$ on $(L_0,h^{L_0})$ by 
\[
\nabla^{L_0}_u=\nabla^L_{d\varphi(Z)(u)}, \quad Z\in T_{x_0}X, \quad u\in T_Z(T_{x_0}X), 
\]
where we use the canonical isomorphism $T_{x_0}X\cong T_Z(T_{x_0}X)$, and a Hermitian connection $\nabla^{E_0}$ on $(E_0,h^{E_0})$ by $\nabla^{E_0}=\varphi^*\nabla^E$. Finally, we set $V^{(x_0)}=\varphi^*V$. 

Let $\Delta^{L_0^p\otimes E_0}$ be the associated Bochner Laplacian acting on $C^\infty(T_{x_0}X,L_0^p\otimes E_0)$.  Introduce the operator $H^{(x_0)}_p$ acting on $C^\infty(T_{x_0}X,L_0^p\otimes E_0)$ by
\begin{equation}\label{e:defHx0p}
H^{(x_0)}_p=\frac{1}{p^2} \Delta^{L_0^p\otimes E_0}+ V^{(x_0)}.
\end{equation}
It is clear that, for any $u \in {C}^\infty_c(B^{T_{x_0}X}(0,\varepsilon))$, we have
\begin{equation}\label{e:Hp=HX0p}
H_pu(Z)=H^{(x_0)}_pu(Z).
\end{equation}

Let $dv^{(x_0)}$ be the Riemannian volume form of $(T_{x_0}X, g^{(x_0)})$ and
\begin{equation}\label{e:4.6}
	K_{\varphi(H^{(x_0)}_p)}\in {C}^{\infty}(T_{x_0}X\times T_{x_0}X,
\pi_1^*(L_0^p\otimes E_0)\otimes \pi_2^*(L^p_0\otimes E_0)^*)
\end{equation}
the Schwartz kernel of the operator $\varphi(H^{(x_0)}_p)$ with respect to the volume form $dv^{(x_0)}$. Recall (see \eqref{e:Kphix0}) that the Schwartz kernel $K_{\varphi(H_p)}$ induces a smooth section $K_{\varphi(H_p),x_0}(Z,Z^\prime)$ of the vector bundle $\pi^*(\operatorname{End}(E))$ on $TX\times_X TX$ defined for all $x_0\in X$ and $Z,Z^\prime\in T_{x_0}X$ with $|Z|, |Z^\prime|<r_X$. 

Using \eqref{e:Hp=HX0p} and the finite propagation speed property of solutions of hyperbolic equations, one can show that the kernels $K_{\varphi(H_p),x_0}(Z,Z^\prime)$ and $K_{\varphi(H^{(x_0)}_p)}(Z,Z^\prime)$ are asymptotically close on $B^{T_{x_0}X}(0,\varepsilon)$ in the $C^\infty$-topology, as $p\to \infty$ (see \cite[Proposition 1.3]{ma-ma08} or \cite[Proposition 2.1]{Bochner-trace}).

\begin{prop}\label{p:Pq-difference}
For any $\varepsilon_1\in (0,\varepsilon)$ and $N\in \NN$, there exists $C>0$ such that 
 \[
|K_{\varphi(H_p),x_0}(Z,Z^\prime)-K_{\varphi(H^{(x_0)}_p)}(Z,Z^\prime)|\leq Cp^{-N}
\]
for any $p\in\NN$, $x_0\in X$  and $Z,Z^\prime\in B^{T_{x_0}X}(0,\varepsilon_1)$.
 \end{prop}

Proposition~\ref{p:Pq-difference} reduces our considerations to the $C^\infty_b$-bounded family $H^{(x_0)}_{p}$ of second order elliptic differential operators acting on $C^\infty(T_{x_0}X, L_0^p\otimes E_0)\cong C^\infty(T_{x_0}X,E_{x_0})$ (parametrized by $x_0\in X$). Note that, once we moved to the Euclidean space, we can consider $p$ to be a real-valued parameter.

In order to write the operator $H^{(x_0)}_{p}$ more explicitly, we choose an orthonormal basis $\mathbf e=\{e_j, j=1,2,\ldots,d\}$ of $T_{x_0}X$. It gives rise to an isomorphism $T_{x_0}X\cong \mathbb R^{d}$. We will consider each $e_j$ as a constant vector field on $T_{x_0}X$ and use the same notation $e_j$ for the corresponding vector field on $\mathbb R^d$: $e_j=\frac{\partial}{\partial Z_j}$.   

The operator $H^{(x_0)}_{p}$ is given by the formula
\begin{equation}\label{e:Hx0p}  
H^{(x_0)}_{p}=-\frac{1}{p^2}\sum_{j,k=1}^{d} g^{jk}(Z)\left[\nabla^{L_0^p\otimes E_0}_{e_j}\nabla^{L_0^p\otimes E_0}_{e_k}- \sum_{\ell=1}^{d}\Gamma^{\ell}_{jk}(Z)\nabla^{L_0^p\otimes E_0}_{e_\ell}\right]+V(Z),
\end{equation}
where $(g^{jk})$ is the inverse matrix of the Riemannian metric $g^{(x_0)}$ on $T_{x_0}X$ and the functions $\Gamma^{\ell}_{jk}$ are defined by the equality 
$\nabla^{TX}_{e_j}e_k=\sum_{\ell=1}^d \Gamma^{\ell}_{jk}e_\ell$ with $\nabla^{TX}$ the Levi-Civita connection of $(T_{x_0}X, g^{(x_0)})$. 

Let $\Gamma^{L_0}=\sum_{j=1}^{d}\Gamma_j^{L_0}(Z)dZ_j$, $\Gamma_j^{L_0}\in C^\infty(\mathbb R^{d})$, and $\Gamma^{E_0}=\sum_{j=1}^{d}\Gamma_j^{E_0}(Z)dZ_j$, $\Gamma_j^{E_0}\in C^\infty(\mathbb R^{d}, \operatorname{End}(E_{x_0}))$, be the connection forms of $\nabla^{L_0}$ and $\nabla^{E_0}$, respectively. Then 
\begin{equation}\label{e:nablaL0E0}  
\nabla^{L_0^p\otimes E_0}_{e_j}=\nabla_{e_j}+p\Gamma^{L_0}_j(Z)+\Gamma^{E_0}_j(Z),
\end{equation}
where $\nabla_{v}$ denotes the usual derivative of a vector-valued function along $v$.

\begin{ex}\label{gauge}
As an example, consider the case when $X$ is the Euclidean space $\mathbb R^d$ with the canonical Euclidean metric, the Hermitian line bundle $(L,h^L)$ is trivial, $(E,h^E)$ is a trivial Hermitian line bundle with a trivial connection $\nabla^E$ and the magnetic potential $\bf A$ written as
\[
{\bf A}= \sum_{j=1}^{d}A_j(x)\,dx_j, \quad x=(x_1,\ldots, x_d)\in \mathbb R^d.
\]
The operator is given by 
\[
H_h=\sum_{j=1}^{d}\left(\frac{h}{i}\frac{\partial}{\partial x_j}-A_j(x)\right)^2 +V(x), \quad h>0.
\]

Fix $x_0\in \mathbb R^d$. First, we introduce new coordinates $Z\in \mathbb R^d$ by the formula $x=x_0+Z$. Actually, these are normal coordinates near $x_0$: $x=\exp_{x_0}(Z)$, and it is convenient to think of $Z$ as a tangent vector to $\mathbb R^d$ at $x_0$. In these coordinates, the operator is written as 
\[
H_h=\sum_{j=1}^{d}\left(\frac{h}{i}\frac{\partial}{\partial Z_j}-A_j(x_0+Z)\right)^2 +V(x_0+Z), \quad h>0.
\]

Next, we consider the trivialization of the trivial line bundle $L$ over $\mathbb R^d$ by parallel transport along the rays $t\in \mathbb R_+\mapsto x_0+tZ\in \mathbb R^d$. In other words, we make a gauge transformation  such that the new magnetic potential 
$$\mathbf A^{(x_0)}=\mathbf A+d\Phi^{(x_0)}=\sum_{j=1}^dA^{(x_0)}_j(Z)dZ_j,$$ 
satisfies the conditions  
\[
\mathbf A^{(x_0)}(0)=0, \quad 
\iota_{\mathcal R}\mathbf A^{(x_0)}(Z)=\sum_{j=1}^dZ_jA^{(x_0)}_j(Z)=0,
\]
where $\mathcal R=\sum_{j=1}^dZ_j\frac{\partial}{\partial Z_j}$ is the radial vector field in $\mathbb R^d$.

This gauge is sometimes called the Fock-Schwinger, Poincar\'e or transverse gauge. In local index theory, it was introduced in \cite{ABP73} and called the synchronous framing. 
It plays a crucial role in the gauge covariant magnetic perturbation theory elaborated by H. Cornean and G. Nenciu (see, for instance, \cite{CN98,CN00,N02}).

In order to find such a gauge transformation, we have to solve the differential equation 
\[
\mathcal R[\Phi^{(x_0)}](Z)=-\iota_{\mathcal R}\mathbf A(x_0+Z),
\]
which gives
\[
\Phi^{(x_0)}(Z)=-\sum_{j=1}^d\int_0^1 A_j(x_0+\tau Z)Z_j\,d\tau.
\]
The potential $\mathbf A^{(x_0)}$ is given by
\[
A^{(x_0)}_j(Z)=\sum_{k=1}^d\left(\int_0^1 B_{kj}(x_0+\tau Z)\tau\,d\tau\right) Z_k, \quad j=1,\ldots,d.
\]

Expanding both parts of the last formula in Taylor series in $Z$, we get the following well-known fact proved in \cite[Appendix II]{ABP73} (see also \cite[Lemma 1.2.4]{ma-ma:book}): for any $r\geq 0$ and $j=1,\ldots,d$, 
\begin{equation}\label{e:ABP73}
\sum_{|\alpha|=r}\partial^\alpha A^{(x_0)}_j(0)\frac{Z^\alpha}{\alpha!}=\frac{1}{r+1}\sum_{|\alpha|=r-1}\sum_{k=1}^d\partial^\alpha B_{kj}(x_0)\,Z_k\frac{Z^\alpha}{\alpha!}.
\end{equation}

Finally, we have the corresponding unitary transformation of operators:
\begin{equation}\label{e:H-equiv}
 H_h=e^{i\Phi^{(x_0)}/h} H^{(x_0)}_\hbar e^{-i\Phi^{(x_0)}/h},
\end{equation}
where
\begin{equation}\label{e:Hx0h-flat}
H^{(x_0)}_h=\sum_{j=1}^{d}\left(\frac{h}{i}\frac{\partial}{\partial Z_j}-A^{(x_0)}_j(Z)\right)^2 +V(x_0+Z), \quad h>0.
\end{equation}

In the current setting, the operators are defined globally, in the entire tangent space $T_{x_0}X\cong \mathbb R^d$, and we don't need the above extension procedure. 
\end{ex}

 \section{Rescaling and formal expansions}\label{scale}
Fix $x_0\in X$. We introduce the rescaling of the operator $H_p^{(x_0)}$ defined by \eqref{e:defHx0p}.  Denote $h=\frac{1}{p}$. For $s\in C^\infty(T_{x_0}X, E_{x_0})$, set
\[
S_hs(Z)=s(Z/h), \quad Z\in T_{x_0}X.
\]
Define the operator $\mathcal H_h$ acting on $s\in C^\infty_c(T_{x_0}X, E_{x_0})$ by
\begin{equation}\label{scaling}
\mathcal H_h=S^{-1}_h\kappa^{\frac 12}H_p^{(x_0)}\kappa^{-\frac 12}S_h,
\end{equation}
where $\kappa=\kappa_{x_0}$ is defined in \eqref{e:def-kappa}. It depends on $x_0$, but, for simplicity of notation, we will omit its dependence on $x_0$. By construction, this is a self-adjoint operator in $L^2(T_{x_0}X, E_{x_0})$, and its spectrum coincides with the spectrum of $H_p^{(x_0)}$. 

If we fix an orthonormal basis $\mathbf e=\{e_j, j=1,2,\ldots,d\}$ of $T_{x_0}X$ for each $x_0$, then we get
a family of second order differential operators on $C^\infty(\RR^d, E_{x_0})$, which depend on the choice of $\mathbf e$ and will be also denoted by $\mathcal H_h$.  For simplicity of notation, we will omit its dependence on $\mathbf e$. By \eqref{e:Hx0p}, the operator $\mathcal H_h$ is given by the formula
\begin{equation}\label{e:Ht}  
\mathcal H_h=-\sum_{j,k=1}^{d} g^{jk}(hZ)\left[\nabla_{h,e_j}\nabla_{h,e_k}-h\sum_{\ell=1}^{d}\Gamma^{\ell}_{jk}(hZ)\nabla_{h,e_\ell}\right]+V(hZ),
\end{equation}
where the rescaled connection $\nabla_h$ is given by
\begin{equation}\label{e:nablat}
\nabla_h=hS^{-1}_h\kappa^{\frac 12}\nabla^{L_0^p\otimes E_0}\kappa^{-\frac 12}S_h.
\end{equation}
By \eqref{e:nablaL0E0}, we have  
\begin{equation}\label{E:nablatej}
\begin{split}
\nabla_{h,e_j}&=h\kappa^{\frac 12}(hZ)\left(\frac{1}{h}\nabla_{e_j}+\frac{1}{h}\Gamma^{L_0}_j(hZ)+\Gamma^{E_0}_j(hZ)\right)\kappa^{-\frac 12}(hZ)\\
&=\nabla_{e_j}+\Gamma^{L_0}_j(hZ)+h\Gamma^{E_0}_j(hZ)-h\left(\kappa^{-1}\frac{\partial \kappa}{\partial Z_j}\right)(hZ).
\end{split}
\end{equation}

The operators $\nabla_{h,e_i}$ depend smoothly on $h$ up to $h=0$, and, since $\Gamma^{L_0}_j(0)=0$,   the limit $\nabla_{0,e_j}$ of $\nabla_{h,e_j}$ as $h\to 0$ coincides with the flat connection $\nabla$:
\begin{equation}\label{E:nabla0ej}
\nabla_{0,e_j}=\nabla_{e_j}=\frac{\partial}{\partial Z_j}. 
\end{equation}
 
Now we expand the coefficients of the operator $\mathcal H_h$ in Taylor series in $h$. For any $m\in \NN$, we get
\begin{equation}\label{e:Ht-formal}
\mathcal H_h=\mathcal H^{(0)}+\sum_{r=1}^m \mathcal H^{(r)}h^r+\mathcal O(h^{m+1}), 
\end{equation}
where there exists $m^\prime\in \NN$ so that for every $k\in\NN$ and $h\in [0,1]$ the derivatives up to order $k$ of the coefficients of the operator $\mathcal O(h^{m+1})$ are bounded by $Ch^{m+1}(1+|Z|)^{m^\prime}$. 

The leading term $\mathcal H^{(0)}$ in \eqref{e:Ht-formal} is given by
\begin{equation}\label{e:H0}  
\mathcal H^{(0)}=-\sum_{j=1}^{d} (\nabla_{e_j})^2+V(x_0).
\end{equation}
 
The next terms $\mathcal H^{(r)}, r\geq 1,$ have the form
\begin{equation}\label{e:Hj}
\mathcal H^{(r)}=\sum_{k,\ell=1}^{d} a_{k\ell,r}\nabla_{e_k}\nabla_{e_\ell}+\sum_{k=1}^{d} b_{k,r}\nabla_{e_k}+c_{r},
\end{equation}
where $a_{k\ell,r}$ is a homogeneous polynomial in $Z$ of degree $r$,  $b_{k,r}$ and $c_{r}$ are polynomials in $Z$ of degree $\leq r$. It is easy to see that the coefficients of $\mathcal H^{(r)}$ are universal polynomials depending on a finite number of derivatives of $V$ and the curvatures $R^{TX}$, $R^L$ and $R^E$ in $x_0$.
 
%
%

\begin{ex}\label{gauge2}
In the setting of Example~\ref{gauge}, by \eqref{e:Hx0h-flat}, we have 
\[
{\mathcal H}^{(x_0)}_h=-\sum_{j=1}^{d}\left(\nabla_{e_j}-iA^{(x_0)}_j(hZ)\right)^2 +V(x_0+hZ), \quad h>0,
\]
where $\nabla_{e_j}=\frac{\partial}{\partial Z_j}$. 

By \eqref{e:ABP73}, we have an asymptotic expansion
\[
A^{(x_0)}_j(hZ)\sim \sum_{r=1}^\infty h^{r}A_{j,r,x_0}(Z), \quad Z\in \mathbb R^d, \quad j=1,\ldots,d, 
\]
where
\[
A_{j,r,x_0}(Z)=\frac{1}{r+1}\sum_{|\alpha|=r-1}\sum_{k=1}^d(\partial^\alpha B_{jk})_{x_0} \,Z_k\frac{Z^\alpha}{\alpha!}
\]
is homogeneous of degree $r$. In particular, we have 
\[
A_{j,1,x_0}(Z)=\frac{1}{2}\sum_{k=1}^dB_{k j}(x_0)\,Z_k, \quad
A_{j,2,x_0}(Z)=\frac{1}{3}\sum_{k,\ell=1}^d\frac{\partial B_{jk}}{\partial x_\ell}(x_0) \,Z_k Z_\ell.
\]

Expanding the coefficients of the operator ${\mathcal H}^{(x_0)}_h$ in Taylor series in $h$, for any $m\in \NN$, we get the expansion \eqref{e:Ht-formal}. The leading term $\mathcal H^{(0)}$ is given by \eqref{e:H0}. The terms $\mathcal H^{(r)}, r\geq 1,$ are first order differential operators of the form
\begin{equation}\label{e:Hj-flat}
\mathcal H^{(r)}=\sum_{j=1}^{d}b_{jr}\frac{\partial}{\partial Z_j}+c_r,
\end{equation}
where $$b_{jr}=2iA_{j,r,x_0}$$ is a polynomial in $Z$ of degree $r$ and $$c_r=\sum_{j=1}^{d}\left(\frac{\partial A_{j,r,x_0}}{\partial x_j} +\sum_{l=1}^{r} A_{j,l,x_0} A_{j,r-l,x_0}\right)+\sum_{|\alpha|=r}\partial^\alpha V(x_0)\frac{Z^\alpha}{\alpha!}$$ is a polynomial in $Z$ of degree $r$.  

In particular, we have 
\begin{equation}\label{e:H1}
\mathcal H^{(1)}=i\sum_{j,k=1}^{d}B_{k j}(x_0)\,Z_k\nabla_{e_j}+\sum_{j=1}^{d}\frac{\partial V}{\partial x_j}(x_0)Z_j,
\end{equation}
and
\begin{equation}\label{e:H2}
\begin{aligned}
\mathcal H^{(2)}=& \frac{2}{3}i\sum_{j,k,\ell=1}^d\frac{\partial B_{jk}}{\partial x_\ell}(x_0) \,Z_k Z_\ell \nabla_{e_j}\\ & +\sum_{j=1}^{d}\left(\frac{1}{3}\sum_{k=1}^d\frac{\partial B_{jk}}{\partial x_j}(x_0)  \,Z_k  +\frac{1}{4}\left(\sum_{k=1}^dB_{k j}(x_0)\,Z_k\right)^2 \right)\\ &+\frac 12 \sum_{j,k=1}^d \frac{\partial^2V}{\partial x_j\partial x_k}(x_0)Z_jZ_k.
\end{aligned}
\end{equation}
\end{ex}

\section{Norm estimates}\label{norm}
In this section, we will establish norm estimates for the operators $\varphi(\mathcal H_{h})$ and its derivatives of any order with respect to $h$.  

Let $H^m=H^m(\mathbb R^{d},E_{x_0})$ be the standard Sobolev space of order $m\in \mathbb Z$ of vector-valued functions on $\mathbb R^{d}$ equipped with the standard Sobolev norm 
\[
\|s\|^2_{m}=\int_{\mathbb R^{d}}(1+|\xi|^2)^{m/2}|\hat s(\xi)|_{h^{E^*_{x_0}}}^2 d\xi,\quad s\in H^m, 
\]
where $\hat s(\xi)$ is the Fourier transform of $s$. In particular, $H^0=L^2(\mathbb R^{d},E_{x_0})$ and 
\[
\|s\|^2_{0}=\int_{\mathbb R^{d}}|s(Z)|_{h^{E_{x_0}}}^2dZ, \quad s\in H^0.
\] 


For $\alpha\in \ZZ^d_+$, we will use the standard notation
\[
Z^\alpha=Z_1^{\alpha_1}Z_2^{\alpha_2}\ldots Z_d^{\alpha_d}, \quad Z\in \RR^d. 
\]
For any $m\in \mathbb Z$ and $M\in \ZZ_+$, we set
\[
\|s\|_{m,M}:= \sum_{|\alpha|\leq M}\left\|Z^{\alpha} s\right\|_{m}, \quad s\in C^\infty_c(\mathbb R^{d},E_{x_0}).
\]



\begin{thm}\label{Thm1.9}
For any $h\in [0,1]$, $m\in \mathbb Z$, $M\in \ZZ_+$, $\lambda=\mu+i\nu\in \CC\setminus \RR$ and $x_0\in X$, we have 
\begin{equation}\label{e:mm+2t}
\left\|(\lambda-\mathcal H_h)^{-1}s\right\|_{m+2,M}\leq C|\nu|^{-M-1}\left\|s\right\|_{m,M},\quad s\in C^\infty_c(\mathbb R^{d},E_{x_0}),  
\end{equation}
where $C=C_{m,M}>0$ is independent of $h\in [0,1]$, $\lambda\in  \CC\setminus \RR$ and $x_0\in X$.
\end{thm}

\begin{proof}
For any $h\in [0,1]$, the operator $\mathcal H_h$ is a second order uniformly elliptic differential operator with coefficients in $C^\infty_b(\mathbb R^d, \operatorname{End}(E_{x_0}))$. Moreover, the operator family $\{\mathcal H_h: h\in [0,1], x_0\in X \}$ is uniformly elliptic and uniformly bounded in $C^\infty_b(\mathbb R^d, \operatorname{End}(E_{x_0}))$. By standard elliptic theory, we infer that this family is a uniformly bounded family of bounded operators from $H^{2m}$ to $L^2$:
\[
\|\mathcal H_h^{m}s\|_0\leq C_m\|s\|_{2m},\quad s\in C^\infty_b(\mathbb R^d, \operatorname{End}(E_{x_0})), h\in [0,1], x_0\in X,
\]
and the standard elliptic estimate holds:
\[
c_m\|s\|_{2m}\leq \|\mathcal H_h^{m}s\|_0+a_m\|s\|_0,\quad s\in C^\infty_b(\mathbb R^d, \operatorname{End}(E_{x_0})), h\in [0,1], x_0\in X. 
\]
Using these estimates and spectral theorem, we proceed as follows:
\begin{multline*}
\|(\lambda-\mathcal H_h)^{-1}s\|_{2m+2}\\
\begin{aligned}
& \leq \frac{1}{c_{m+1}}(\|\mathcal H_h^{m+1}(\lambda-\mathcal H_h)^{-1} s\|_0+a_{m+1}\|(\lambda-\mathcal H_h)^{-1} s\|_0)\\ & \leq \frac{1}{c_{m+1}}(\|\mathcal H_h(\lambda-\mathcal H_h)^{-1} s\|_0\|\mathcal H_h^{m}s\|_0+a_{m+1}\|(\lambda-\mathcal H_h)^{-1} s\|_0)\\ & \leq \frac{C_m}{c_{m+1}}\sup_{t\in \mathbb R}\left|t(\lambda-t)^{-1}\right|\|s\|_{2m}+ \frac{a_{m+1}}{c_{m+1}}|\nu|^{-1} \|s\|_0\leq C|\nu|^{-1}\|s\|_{2m},
\end{aligned}
\end{multline*}
that proves \eqref{e:mm+2t} for any $m\geq 0$. The case of arbitrary $m\in \mathbb Z$ follows by duality and interpolation.

Now we proceed by induction on $M$. We assume that \eqref{e:mm+2t} holds for all $m\in \ZZ$ and some $M\geq 0$. Take $\alpha\in \ZZ^d_+$ with $|\alpha|=M+1$. Then we write
\begin{multline}\label{e:est1}
Z^{\alpha} (\lambda-\mathcal H_{h})^{-1} s
= (\lambda-\mathcal H_{h})^{-1}Z^{\alpha}s+[Z^{\alpha},(\lambda-\mathcal H_{h})^{-1}]s\\
=(\lambda-\mathcal H_{h})^{-1}Z^{\alpha}s +(\lambda-\mathcal H_{h})^{-1}[Z^{\alpha}, \mathcal H_{h}](\lambda-\mathcal H_{h})^{-1}.
\end{multline}
By \eqref{e:Ht}, the commutator $[Z^{\alpha}, \mathcal H_{h}]$ is a first order differential operator given by
\begin{multline}\label{e:comm-Hh}
[Z^{\alpha},\mathcal H_h] =-\sum_{j,k=1}^{d} g^{jk}(hZ)\Big[[Z^{\alpha},\nabla_{h,e_j}]\nabla_{h,e_k}+\nabla_{h,e_j}[Z^{\alpha},\nabla_{h,e_k}]\\ - h\sum_{\ell=1}^{d}\Gamma^{\ell}_{jk}(hZ)[Z^{\alpha},\nabla_{h,e_\ell}]\Big].
\end{multline}

Using the commutation relation
\begin{equation}\label{e:Z-nabla-com}
[Z_\ell,\nabla_{h;e_j}]=-\delta_{j\ell},\quad  j,\ell=1,2,\ldots,d,
\end{equation}
we get for any $\alpha^\prime\in \ZZ^d_+$
\begin{equation}\label{e:Z-nabla1-com}
[Z^{\alpha^\prime},\nabla_{h;e_j}]=-\alpha^\prime_j Z^{\alpha^\prime-\delta_j},\quad  j=1,2,\ldots,d, 
\end{equation}
where $\delta_j\in \ZZ^d_+$ is given by $(\delta_j)_\ell=\delta_{\ell j}, \ell=1,2,\ldots,d$. 
Therefore, for any $h\in (0,1]$ and $m\in \mathbb Z$,  we have
\begin{equation}\label{e:est-comm}
\|[Z^\alpha,\nabla_{h;e_j}]s\|_{m+2}\leq C_1\|s\|_{m+2,M}
\end{equation}
and 
\begin{multline}\label{e:est-nablat}
\|\nabla_{h;e_j} s\|_{m,M}= \sum_{|\alpha^\prime|\leq M}\left\|Z^{\alpha^\prime} \nabla_{h;e_j} s\right\|_{m+1}\leq \sum_{|\alpha^\prime|\leq M}\left\|\nabla_{h;e_j} Z^{\alpha^\prime} s\right\|_{m}\\ + \sum_{|\alpha^\prime|\leq M}\left\|[Z^{\alpha^\prime}, \nabla_{h;e_j}] s\right\|_{m} \leq C_2\|s\|_{m+1,M}, \quad s\in C^\infty_c(\mathbb R^{d},E_{x_0}),
\end{multline}
where $C_1, C_2>0$ are independent of $h\in (0,1]$ and $x_0\in X$.  

By \eqref{e:comm-Hh} \eqref{e:est-comm} and \eqref{e:est-nablat}, we conclude that, for any $h\in (0,1]$ and $m\in \mathbb Z$,
\begin{equation}\label{e:est-comm-res}
\|[Z^{\alpha},\mathcal H_h]s\|_{m}\leq C \|s\|_{m+1,M}.
\end{equation}

Using the induction hypothesis and \eqref{e:est-comm-res}, we proceed in \eqref{e:est1} as follows: 
\begin{align*}
\left\|Z^{\alpha} (\lambda-\mathcal H_{h})^{-1} s\right\|_{m+2} 
\leq & C |\nu|^{-1}  \left(\left\|s\right\|_{m,M+1}+\left\|[Z^{\alpha}, \mathcal H_{h}] (\lambda-\mathcal H_{h})^{-1} s\right\|_{m}\right)\\
\leq & C |\nu|^{-1} \Big(\left\|s\right\|_{m,M+1}+\left\|(\lambda-\mathcal H_{h})^{-1} s\right\|_{m+1,M} \Big)\\
\leq & C |\nu|^{-M-2}\|s\|_{m,M+1},
\end{align*}
that completes the proof of \eqref{e:mm+2t}. 
\end{proof}

By the Dynkin-Helffer-Sj\"{o}strand formula, for $\varphi\in C^\infty_c(\RR)$, we have  
\begin{equation}\label{e:HS}
\varphi(\mathcal H_{h})=-\frac{1}{\pi }
\int_\CC \frac{\partial \tilde{\varphi}}{\partial \bar \lambda}(\lambda)(\lambda-\mathcal H_{h})^{-1}d\mu d\nu,
\end{equation}
where $\tilde{\varphi}\in C^\infty_c(\CC)$ is an almost-analytic extension of $\varphi$ satisfying 
\[
\frac{\partial \tilde{\varphi}}{\partial \bar \lambda}(\lambda)=O(|\nu|^\ell),\quad \lambda=\mu+i\nu, \quad \nu\to 0,
\]
for any $\ell\in \NN$.

Let us consider the function $\psi(\lambda)=\varphi(\lambda)(a-\lambda)^K$ with any $K\in \NN$ and $a<0$. Then its almost-analytic extension $\tilde{\psi}$ can be taken to be $\tilde{\psi}(\lambda)=\tilde{\varphi}(\lambda)(a-\lambda)^K$. If we apply the formula \eqref{e:HS} to $\psi$, then we get 
\begin{equation}\label{e:HS1}
\varphi(\mathcal H_{h})=-\frac{1}{\pi }
\int_\CC \frac{\partial \tilde{\varphi}}{\partial \bar \lambda}(\lambda) (a-\lambda)^K (\lambda-\mathcal H_{h})^{-1}(a-\mathcal H_{h})^{-K}d\mu d\nu.
\end{equation}

\begin{prop}
For any $h\in (0,1]$ and $m, m^\prime\in \ZZ$, the operator $\varphi(\mathcal H_{h})$ extends to a bounded operator from $H^m$  to $H^{m^\prime}$ with the following norm estimate for any $M\in \mathbb Z_+$:
\begin{equation}\label{e:varphi-Ht}
\left\|\varphi(\mathcal H_{h})s\right\|_{m^\prime,M}\leq C\left\|s\right\|_{m,M}, \quad s\in C^\infty_c(\mathbb R^{d},E_{x_0}),
\end{equation}
where $C=C_{M,m.m^\prime}>0$ is independent of $h\in (0,1]$ and $x_0\in X$.  
\end{prop}

\begin{proof}
 By Theorem \ref{Thm1.9}, for any $m^\prime, K\in \NN$ and $M\in \mathbb Z_+$, there exists $C>0$ such that, for all $h\in (0,1]$ and $\lambda\in\CC\setminus \RR$,
\begin{multline*}
\left\|(\lambda - \mathcal H_{h})^{-1} (a - \mathcal H_{h})^{-K}s\right\|_{m^\prime,M}\\ \leq C|\nu|^{-M-1}\left\|s\right\|_{m^\prime-2(K+1),M}, \quad s\in C^\infty_c(\mathbb R^{d},E_{x_0}). 
\end{multline*}
By the above estimates, the desired statement follows immediately from the formula \eqref{e:HS1} with  appropriate $K$.
\end{proof}

\begin{thm}\label{est-rem}
For any $r\geq 0$, $m, m^\prime\in \ZZ$, and $M\in \mathbb Z_+$, there exists $C>0$ such that, for any $h\in (0,1]$, 
\[
\left\|\frac{\partial^r}{\partial h^r}\varphi(\mathcal H_{h}) s\right\|_{m,M}\leq C\|s\|_{m^\prime,M+2r}, \quad s\in C^\infty_c(\mathbb R^{d},E_{x_0}).
\] 
\end{thm}

\begin{proof}
We proceed as in the proof of \cite[Theorem 1.10]{ma-ma08}. Taking $r$th derivatives in $h$ of both parts of \eqref{e:HS1}, we get
\begin{multline}\label{e:diff-phi}
\frac{\partial^r}{\partial h^r}\varphi(\mathcal H_{h})\\ =-\frac{1}{\pi }
\int_\CC \frac{\partial \tilde{\varphi}}{\partial \bar \lambda}(a-\lambda)^K \frac{\partial^r}{\partial h^r}\left[(\lambda-\mathcal H_{h})^{-1}(a-\mathcal H_{h})^{-K}\right] d\mu d\nu.
\end{multline}
We set 
\[
I_{r}=\left\{ {\mathbf r}=(r_1,\ldots,r_j) : \sum_{i=1}^jr_i=r, r_i\in \mathbb N \right\}. 
\]
Then we can write
\begin{multline}\label{diff}
\frac{\partial^r}{\partial h^r}(\lambda-\mathcal H_{h})^{-1}\\ =\sum_{{\mathbf r}\in I_{r}} \frac{r!}{{\mathbf r}!} (\lambda - \mathcal H_{h})^{-1} \frac{\partial^{r_1}\mathcal H_{h}}{\partial h^{r_1}}(\lambda - \mathcal H_{h})^{-1}\cdots  \frac{\partial^{r_j}\mathcal H_{h}}{\partial h^{r_j}}(\lambda - \mathcal H_{h})^{-1}.
\end{multline}
 
Observe that, for $g\in C^\infty_b(\RR^{d})$ and $r\geq 0$, the function $\frac{\partial^r}{\partial h^r}(g(hZ))$  has the form 
$$
\frac{\partial^r}{\partial h^r}(g(hZ))=\sum_{|\beta|\leq r}g_\beta (hZ)Z^\beta
$$ 
with some $g_\beta\in C^\infty_b(\RR^{d})$. 
By \eqref{E:nablatej}, for $r>0$, the operator $\frac{\partial^r}{\partial h^r}\nabla_{h,e_j}$ has the form $$\frac{\partial^r}{\partial h^r}\nabla_{h,e_j}=\sum_{|\beta|\leq r}(f_\beta(hZ)+hg_\beta(hZ))Z^\beta$$ with some $f_\beta, g_\beta\in C^\infty_b(\RR^{d})$. 

Using these facts, \eqref{e:Ht} and \eqref{e:Z-nabla-com}, one can easily see that for any $r$ the operator $\frac{\partial^{r}\mathcal H_{h}}{\partial h^{r}}$ is a second order differential operator of the form
\begin{equation}
\begin{aligned}\label{e:partial-rLt}
\frac{\partial^{r}\mathcal H_{h}}{\partial h^{r}}= & \sum_{j,k=1}^{d}\sum_{|\beta|\leq r} A^{jk}_\beta (hZ)Z^\beta \nabla_{h;e_j}\nabla_{h;e_k}\\ & +\sum_{\ell=1}^{d}\sum_{|\beta|\leq r} (B^{\ell,0}_{\beta}(hZ)+hB^{\ell,1}_{\beta}(hZ))Z^\beta\nabla_{h,e_\ell}\\ & +\sum_{|\beta|\leq r} (C^0_\beta(hZ)+h C^1_\beta(hZ)+h^2 C^2_\beta(hZ)) Z^\beta
\end{aligned}
\end{equation}
with some $A^{jk}_\beta, B^{\ell,0}_{\beta}, B^{\ell,1}_{\beta}, C^0_\beta, C^1_\beta, C^2_\beta\in C^\infty_b(\RR^{d})$.

It follows that, for any $m\in \NN$ and $M\in \mathbb Z_+$, there exists $C>0$ such that, for all $h\in (0,1]$ and $x_0\in X$,  
\begin{equation}\label{e:partial-rL}
\left\|\frac{\partial^{r}\mathcal H_{h}}{\partial h^{r}} s\right\|_{m,M}\leq C\left\|s\right\|_{m+2,M+r}, \quad s\in C^\infty_c(\mathbb R^{d},E_{x_0}). 
\end{equation}

Using \eqref{diff}, \eqref{e:mm+2t} and \eqref{e:partial-rL}, we obtain that for any $m\in \ZZ$ and $M\in \mathbb Z_+$, there exists $C>0$ such that for $\lambda\in \CC\setminus\RR$, $h\in (0,1]$ and $x_0\in X$, 
\begin{multline}\label{e:res-mm-prime}
\left\| \frac{\partial^r}{\partial h^r}\left[(\lambda-\mathcal H_{h})^{-1}(a-\mathcal H_{h})^{-K}\right] s\right\|_{m,M}\\ \leq C_{m,M}|\nu|^{-(M+r+1)}\left\|s\right\|_{m-2(K+1),M+2r}, \quad s\in C^\infty_c(\mathbb R^{d},E_{x_0}). 
\end{multline} 
Using \eqref{e:diff-phi} with appropriate $K$ and \eqref{e:res-mm-prime}, we complete the proof. 
\end{proof}

\section{Proof of the main theorem}\label{proof}
\subsection{Existence of the asymptotic expansion}
By the mean-value theorem and Theorem \ref{est-rem}, for any $s\in C^\infty_c(\mathbb R^{d},E_{x_0})$ and $r\geq 0$, there exists the limit 
\begin{equation}\label{e:defFr}
\lim_{h\to 0}\frac{1}{r!}\frac{\partial^r}{\partial h^r}\varphi(\mathcal H_{h}) s =F_{r}s
\end{equation}
in the norm $\|\cdot\|_{m,M}$ with any $m\in \ZZ$ and $M\in \mathbb Z_+$. Moreover, the limit $F_{r}s$ satisfies the estimates
\[
\left\|F_{r} s\right\|_{m,M}\leq C\|s\|_{m^\prime,M+2r}, \quad s\in C^\infty_c(\mathbb R^{d},E_{x_0}).
\] 
for any $r\geq 0$, $m, m^\prime\in \ZZ$, and $M\in \mathbb Z_+$, where $C>0$ is independent of $s$ and $x_0$. 
 
Next, we convert norm estimates of operators into pointwise estimates of the Schwartz kernels, using the Sobolev embedding theorem. We use the following theorem (cf. \cite[Theorem 5.1]{Bochner-trace}).  
 
\begin{thm}\label{t:pointwise}
Assume that $\{K_h : C^\infty_c(\mathbb R^{d},E_{x_0}) \to C^\infty(\mathbb R^{d},E_{x_0}) : h\in (0,1]\}$ is a family of operators with smooth kernel $K_h(Z,Z^\prime)$ such that there exists $K>0$ such that, for any $m, m^\prime \in \NN$ and $M\in \mathbb Z_+$, there exists $C>0$ such that, for any $h\in (0,1]$, 
\begin{equation}\label{e:est-Kt}
\|K_hs\|_{m,M}\leq C\|s\|_{-m^\prime ,M+K}, \quad s\in C^\infty_c(\mathbb R^{d},E_{x_0}).
\end{equation}
Then, for any $m\in \mathbb N$, there exists $M^\prime>0$ such that for any $N\in \NN$ there exists $C>0$ such that, for any $h\in (0,1]$ and $Z,Z^\prime\in \RR^{d}$ 
\[
\sup_{|\alpha|+|\alpha^\prime|\leq m}\Bigg|\frac{\partial^{|\alpha|+|\alpha^\prime|}}{\partial Z^\alpha\partial Z^{\prime\alpha^\prime}}K_{h}(Z,Z^\prime)\Bigg| 
\leq C(1+|Z|+|Z^\prime|)^{M^\prime}. 
\]
\end{thm}

Let $K_{\varphi(\mathcal H_{h})}(Z,Z^\prime)=K_{\varphi(\mathcal H_{h,x_0})}(Z,Z^\prime)$ be the smooth Schwartz kernel of the operator $\varphi(\mathcal H_{h})$ with respect to $dZ$. 
Let $F_r(Z,Z^\prime)$ be the smooth Schwartz kernel of the operator $F_r : C^\infty_c(\mathbb R^{d},E_{x_0})\to C^\infty(\mathbb R^{d},E_{x_0})$.
From the Taylor formula 
\[
\varphi(\mathcal H_{h})-\sum_{r=0}^j F_{r} h^r=\frac{1}{j!}\int_0^h(h-\tau)^j\frac{\partial^{j+1}\varphi(\mathcal H_{h})}{\partial h^{j+1}}(\tau) d\tau, \quad h\in [0,1],  
\]
Theorems \ref{est-rem} and \ref{t:pointwise}, we immediately get the following. 

\begin{thm}\label{t:thm7.2}
For any $j,m\in \mathbb N$, there exists $M>0$ such that, for any $N\in \NN$,  there exists $C>0$ such that for any $h\in (0,1]$, $Z,Z^\prime\in \RR^{d}$ and $x_0\in X$, 
\begin{multline*}
\sup_{|\alpha|+|\alpha^\prime|\leq m}\Bigg|\frac{\partial^{|\alpha|+|\alpha^\prime|}}{\partial Z^\alpha\partial Z^{\prime\alpha^\prime}}\Bigg(K_{\varphi(\mathcal H_{h})}(Z,Z^\prime)
-\sum_{r=0}^jF_{r}(Z,Z^\prime)h^r\Bigg)\Bigg| \\
\leq Ch^{j+1}(1+|Z|+|Z^\prime|)^{M}. 
\end{multline*}
\end{thm}

By \eqref{scaling}, we have
\[
K_{\varphi(H^{(x_0)}_p)}(Z,Z^\prime)=h^{-d}\kappa^{-\frac 12}(Z)K_{\varphi(\mathcal H_{h})}(Z/h,Z^\prime/h)\kappa^{-\frac 12}(Z^\prime), \quad Z,Z^\prime \in \mathbb R^{d},
\]
that completes the proof of the asymptotic expansion \eqref{e:main-exp} in Theorem~\ref{t:main}.

\subsection{General formulas for the coefficients}\label{s:computation}
In this section, we complete the proof of Theorem~\ref{c:main} by deriving the formulas \eqref{e:f0-d} and \eqref{e:fr2-form} for the coefficients in asymptotic expansion \eqref{e:main-exp} in Theorem~\ref{t:main}. 
  
Since $\lim_{h\to 0}\frac{\partial^{r}\mathcal H_{h}}{\partial h^{r}}=r!\mathcal H^{(r)}$, from \eqref{diff}, we get
\begin{multline}\label{diff0}
\lim_{h\to 0}\frac{\partial^r}{\partial h^r}(\lambda-\mathcal H_{h})^{-1}\\ =\sum_{{\mathbf r}\in I_{r}} {r}!(\lambda - \mathcal H^{(0)})^{-1} \mathcal H^{(r_1)}(\lambda - \mathcal H^{(0)})^{-1}\cdots  \mathcal H^{(r_j)}(\lambda - \mathcal H^{(0)})^{-1}.
\end{multline}
By \eqref{e:diff-phi} with $K=0$ and \eqref{e:defFr}, we infer that
\begin{multline}\label{e:Fr-phi}
F_r= -\frac{1}{\pi }\sum_{{\mathbf r}\in I_{r}}\int_\CC \frac{\partial \tilde{\varphi}}{\partial \bar \lambda}(\lambda) (\lambda - \mathcal H^{(0)})^{-1}\\ \times \mathcal H^{(r_1)}(\lambda - \mathcal H^{(0)})^{-1}\cdots  \mathcal H^{(r_j)}(\lambda - \mathcal H^{(0)})^{-1} d\mu d\nu.
\end{multline}

For $r=0$, we immediately get
\[
F_0= -\frac{1}{\pi }\int_\CC \frac{\partial \tilde{\varphi}}{\partial \bar \lambda}(\lambda) (\lambda - \mathcal H^{(0)})^{-1} d\mu d\nu=\varphi(\mathcal H^{(0)}),
\]
which shows that the leading coefficient $F_{0,x_0}$ in the asymptotic expansion \eqref{e:main-exp} is the Schwartz kernel $K_{\varphi(\mathcal H^{(0)})}$ of the corresponding function $\varphi(\mathcal H^{(0)})$ of the operator $\mathcal H^{(0)}$:
\begin{equation}\label{e:F0}
F_{0,x_0}(Z,Z^\prime)=K_{\varphi(\mathcal H^{(0)})}(Z,Z^\prime).
\end{equation}
as a consequence, we get 
\begin{equation}\label{e:f0}
f_{0}(x_0)=K_{\varphi(\mathcal H^{(0)})}(0,0).
\end{equation}

Using separation of variables and the Fourier transform, we infer that
\begin{equation}\label{e:Phi-kernel}
K_{\varphi (\mathcal H^{(0)})}(Z,Z^\prime)=\frac{1}{(2\pi)^{d}} \int_{\RR^{d}} e^{i(Z-Z^\prime)\xi}\varphi(|\xi|^2+V(x_0))d\xi
\end{equation}
and
\[
K_{\varphi (\mathcal H^{(0)})}(Z,Z)=\frac{1}{(2\pi)^{d}} \int_{\RR^{d}} \varphi(|\xi|^2+V(x_0))d\xi,
\]
which gives \eqref{e:f0-d} by \eqref{e:f0}.

Now we compute lower order coefficients.  We have
\[
\nabla_{e_m}\mathcal H^{(0)}=\mathcal H^{(0)}\nabla_{e_m},
\]
which implies that
\begin{equation}
\nabla_{e_m}(\lambda-\mathcal H^{(0)})^{-1}=(\lambda-\mathcal H^{(0)})^{-1}\nabla_{e_m}.
\label{e:comm-res1}
\end{equation}
For any polynomial $g(Z)$ with values in $\operatorname{End}(E_{x_0})$,  we have
\[
g\mathcal H^{(0)}=\mathcal H^{(0)}g+2\sum_{j=1}^d\frac{\partial g}{\partial Z_j}\nabla_{e_j}-\Delta g
\]
which implies that
\begin{multline}
(\lambda-\mathcal H^{(0)})^{-1}g= g(\lambda-\mathcal H^{(0)})^{-1}\\
+ (\lambda-\mathcal H^{(0)})^{-1}  \left(-2\sum_{j=1}^d\frac{\partial g}{\partial Z_j}\nabla_{e_j}+\Delta g\right) (\lambda-\mathcal H^{(0)})^{-1}.
\label{e:comm-g-res}
\end{multline}
By \eqref{e:comm-res1} and \eqref{e:comm-g-res}, it follows that for any differential operator $A$ of order $s$ with polynomial coefficients of degree $\leq r$, we have 
%
%
\[
(\lambda-\mathcal H^{(0)})^{-1}A= A(\lambda-\mathcal H^{(0)})^{-1} 
+ (\lambda-\mathcal H^{(0)})^{-1}  A_1 (\lambda-\mathcal H^{(0)})^{-1}.
\]
where $A_1$ is a differential operator of order $s+1$ with polynomial coefficients of degree $\leq r-1$. 
 
Iterating this formula, after finitely many steps, we will arrive at an expression of the form
\[
(\lambda-\mathcal H^{(0)})^{-1}A=A(\lambda-\mathcal H^{(0)})^{-1} 
+ \sum_{j=1}^rA_{j}(\lambda-\mathcal H^{(0)})^{-j-1},
\] 
where $A_j$, $j=1,\ldots,r$, is a differential operator of order $s+j$ with polynomial coefficients of degree $\leq r-j$.  Let us introduce a grading on the algebra of differential operators on $C^\infty(\RR^d,E_{x_0})$ with polynomial coefficients, setting the weight of each $Z_j$ and $\frac{\partial}{\partial Z_k}$ equal to $1$. Then the weight of each $A_j$, $j=1,\ldots,r$ does not exceed the weight of $A$. 


By \eqref{e:Hj}, the coefficients of $\mathcal H^{(r)}$ are polynomials of weight $r+2$. It follows that 
\begin{multline*}
(\lambda - \mathcal H^{(0)})^{-1} \mathcal H^{(r_1)}(\lambda - \mathcal H^{(0)})^{-1}\cdots  \mathcal H^{(r_j)}(\lambda - \mathcal H^{(0)})^{-1} \\ =\sum_{\ell=1}^{N(\mathbf r)} D_{\ell,\mathbf r}(\lambda-\mathcal H^{(0)})^{-\ell},
\end{multline*} 
where $D_{\ell,\mathbf r}$ is a differential operator of weight $\leq r+2j$. 

By this formula and \eqref{e:Fr-phi}, we conclude that
\[
F_r= -\frac{1}{\pi }\sum_{{\mathbf r}\in I_{r}}  \sum_{\ell=1}^{N(\mathbf r)} D_{\ell,\mathbf r}  \int_\CC \frac{\partial \tilde{\varphi}}{\partial \bar \lambda}(\lambda)(\lambda -\mathcal H^{(0)})^{-\ell} d\mu d\nu.
\]
By differentiating the Dynkin-Helffer-Sj\"{o}strand formula \eqref{e:HS}, we get  
\begin{equation}\label{e:phik}
\varphi^{(k)}(\mathcal H^{(0)})=-\frac{k!}{\pi }
\int_\CC \frac{\partial \tilde{\varphi}}{\partial \bar \lambda}(\lambda)(\lambda-\mathcal H^{(0)})^{-k-1}d\mu d\nu.
\end{equation}
It follows that
\[
F_r=\sum_{\ell=1}^{N_r} D_{\ell,r} \varphi^{(\ell-1)}(\mathcal H^{(0)}),
\]
where $D_{\ell,r}=\frac{1}{(\ell-1)!}\sum_{{\mathbf r}\in I_{r}} D_{\ell,\mathbf r}$ is a differential operator of weight $3r$ and $N_r=\max_{\mathbf r\in I_r} N(\mathbf r)$. If we write 
$$
D_{\ell,r}=\sum_{|\alpha|\leq 3r} c_{\ell,r,\alpha}(x_0)\nabla^\alpha 
$$
with some $c_{\ell,r,\alpha}(x_0)\in \operatorname{End}(E_{x_0})$, we get
\begin{equation}\label{e:fr2}
f_r(x_0)=  F_r(0,0)=\sum_{\ell=1}^{N_r}\sum_{|\alpha|\leq 3r}  c_{\ell,r,\alpha}(x_0)  K_{\nabla^\alpha\varphi^{(\ell-1)}(\mathcal H^{(0)})}(0,0).
\end{equation}
 By \eqref{e:Phi-kernel}, we have
\begin{equation}\label{e:nablaPhi-kernel}
K_{\nabla^\alpha \varphi (\mathcal H^{(0)})}(Z,Z^\prime)=\frac{1}{(2\pi)^{d}} \int_{\RR^{d}} e^{i(Z-Z^\prime)\xi}(i\xi)^\alpha \varphi(|\xi|^2+V(x_0))d\xi.
\end{equation}
By \eqref{e:nablaPhi-kernel} and \eqref{e:fr2}, we get \eqref{e:fr2-form} with 
\[
P_{r,\ell,x_0}(\xi)=\sum_{|\alpha|\leq 3r}  c_{\ell,r,\alpha}(x_0)(i\xi)^\alpha.  
\]
Since the coefficients of $\mathcal H^{(r)}$ are universal polynomials depending on a finite number of derivatives of $V$ and the curvatures $R^{TX}$, $R^L$ and $R^E$ in $x_0$, it is easy to see that the same is true for the coefficients of $D_{\ell,\mathbf r}$ and $D_{\ell,r}$, that is, for $c_{\ell,r,\alpha}(x_0)$. 

\subsection{Explicit computations of the coefficients}\label{s:ex-computation}
In this section, we derive explicit formulas for the coefficients $f_1(x_0)$ and $f_2(x_0)$ in the setting of Examples~\ref{gauge} and~\ref{gauge2} (cf. \cite[Theorem 1]{CdV12}):
\begin{equation}\label{e:exf1}
f_1(x_0)=0.
\end{equation}
\begin{multline}\label{e:exf2}
f_{2}(x_0)=-\frac{1}{12}\sum_{j=1}^d\left(\frac{\partial V}{\partial x_j}(x_0)\right)^2 \frac{1}{(2\pi)^{d}} \int_{\RR^{d}}\varphi^{\prime\prime\prime}(|\xi|^2+V(x_0))d\xi \\
-\frac{1}{6} \left(\frac{1}{2}\sum_{j,k=1}^dB_{k j}(x_0)^2 +\sum_{j=1}^d\frac{\partial^2 V}{\partial x^2_j}(x_0)\right)\frac{1}{(2\pi)^{d}} \int_{\RR^{d}}\varphi^{\prime\prime}(|\xi|^2+V(x_0))d\xi.
\end{multline}

By \eqref{e:Fr-phi}, we have 
\begin{equation}\label{e:Fr-phi1}
F_1= -\frac{1}{\pi }\int_\CC \frac{\partial \tilde{\varphi}}{\partial \bar \lambda}(\lambda)  (\lambda - \mathcal H^{(0)})^{-1}\mathcal H^{(1)}(\lambda - \mathcal H^{(0)})^{-1} d\mu d\nu.
\end{equation}

If $g$ is a polynomial of degree $\leq 1$, then, by \eqref{e:comm-g-res}, we have
\begin{equation}\label{e:comm-g-res1}
(\lambda-\mathcal H^{(0)})^{-1}g
= g(\lambda-\mathcal H^{(0)})^{-1}
-2\sum_{j=1}^d\frac{\partial g}{\partial Z_j}\nabla_{e_j} (\lambda-\mathcal H^{(0)})^{-2}.
\end{equation}

By \eqref{e:H1} and \eqref{e:comm-g-res1}, using the fact that the matrix $(B_{kj})$ is skew-symmetric, we have
\begin{multline}\label{e:H0H1}
(\lambda-\mathcal H^{(0)})^{-1}\mathcal H^{(1)}  
= \mathcal H^{(1)} (\lambda-\mathcal H^{(0)})^{-1}\\
-2\left[\sum_{j,k=1}^{d}iB_{k j}(x_0)\,\nabla_{e_k}\nabla_{e_j}+\sum_{j=1}^d\frac{\partial V}{\partial x_j}(x_0)\nabla_{e_j}\right] (\lambda-\mathcal H^{(0)})^{-2}\\
= \mathcal H^{(1)}  (\lambda-\mathcal H^{(0)})^{-1}
-2\sum_{j=1}^d\frac{\partial V}{\partial x_j}(x_0)\nabla_{e_j} (\lambda-\mathcal H^{(0)})^{-2}. 
\end{multline}
By \eqref{e:Fr-phi1} and \eqref{e:phik}, we get
\[
F_1=\mathcal H^{(1)}\varphi^\prime (\mathcal H^{(0)})\\
-\sum_{j=1}^d\frac{\partial V}{\partial x_j}(x_0)\nabla_{e_j} \varphi^{\prime\prime} (\mathcal H^{(0)}).
\]
Now we consider the corresponding Schwartz kernels and set $Z=Z^\prime=0$.
By \eqref{e:nablaPhi-kernel}, for any $\psi\in C^\infty_c(\mathbb R^d)$, we have 
\begin{equation}\label{e:nablaejPhi-kernel}
K_{\nabla_{e_j} \psi (\mathcal H^{(0)})}(0,0)=\frac{1}{(2\pi)^{d}} \int_{\RR^{d}} i\xi_j \psi(|\xi|^2+V(x_0))d\xi=0,
\end{equation}
since the integrand is an odd function. It follows that
\[
f_1(x_0)=F_1(0,0)=-\sum_{j=1}^d\frac{\partial V}{\partial x_j}(x_0)K_{\nabla_{e_j}\varphi^{\prime\prime} (\mathcal H^{(0)})}(0,0)=0,
\]
that proves \eqref{e:exf1}.

By \eqref{e:Fr-phi}, we have 
\[
F_2=F_{2,1}+F_{2,2},
\]
where
\begin{align}
F_{2,1}=&-\frac{1}{\pi }\int_\CC \frac{\partial \tilde{\varphi}}{\partial \bar \lambda}(\lambda)  (\lambda - \mathcal H^{(0)})^{-1}\mathcal H^{(1)}(\lambda - \mathcal H^{(0)})^{-1}\mathcal H^{(1)}(\lambda - \mathcal H^{(0)})^{-1} d\mu d\nu,\label{e:F21}\\
F_{2,2}=&-\frac{1}{\pi }\int_\CC \frac{\partial \tilde{\varphi}}{\partial \bar \lambda}(\lambda)  (\lambda - \mathcal H^{(0)})^{-1}\mathcal H^{(2)}(\lambda - \mathcal H^{(0)})^{-1} d\mu d\nu. \label{e:F22}
\end{align}

Taking the adjoints in \eqref{e:H0H1}, we get
\begin{multline}\label{e:H1H0}
\mathcal H^{(1)} (\lambda-\mathcal H^{(0)})^{-1}  
\\
=  (\lambda-\mathcal H^{(0)})^{-1}\mathcal H^{(1)} 
+2(\lambda-\mathcal H^{(0)})^{-2}\sum_{j=1}^d\frac{\partial V}{\partial x_j}(x_0)\nabla_{e_j}. 
\end{multline}
Now we use \eqref{e:H0H1} and \eqref{e:H1H0}:
\begin{multline}\label{e:H0H1H0H1H0}
 (\lambda - \mathcal H^{(0)})^{-1}\mathcal H^{(1)}(\lambda - \mathcal H^{(0)})^{-1}\mathcal H^{(1)}(\lambda - \mathcal H^{(0)})^{-1}\\
 \begin{aligned}
=&\mathcal H^{(1)}  (\lambda-\mathcal H^{(0)})^{-4}\mathcal H^{(1)}+2\mathcal H^{(1)}  (\lambda-\mathcal H^{(0)})^{-2}\sum_{j=1}^d\frac{\partial V}{\partial x_j}(x_0)\nabla_{e_j}\\
&-2\sum_{j=1}^d\frac{\partial V}{\partial x_j}(x_0)\nabla_{e_j} (\lambda-\mathcal H^{(0)})^{-4}\mathcal H^{(1)}\\
&-4\sum_{j=1}^d\frac{\partial V}{\partial x_j}(x_0)\nabla_{e_j} (\lambda-\mathcal H^{(0)})^{-5}\sum_{k=1}^d\frac{\partial V}{\partial x_k}(x_0)\nabla_{e_k}.
\end{aligned}
\end{multline}
By \eqref{e:F21}, \eqref{e:H0H1H0H1H0} and \eqref{e:phik}, we get
\begin{align*}
F_{2,1} =& \frac{1}{2}\mathcal H^{(1)} \varphi^{\prime\prime}(\mathcal H^{(0)})\mathcal H^{(1)}+\frac{2}{3}\mathcal H^{(1)}  \varphi^{\prime\prime\prime}(\mathcal H^{(0)})\sum_{j=1}^d\frac{\partial V}{\partial x_j}(x_0)\nabla_{e_j}\\
& -\frac{2}{3}\sum_{j=1}^d\frac{\partial V}{\partial x_j}(x_0)\nabla_{e_j} \varphi^{\prime\prime\prime}(\mathcal H^{(0)})\mathcal H^{(1)}\\
& -\frac{1}{6}\sum_{j=1}^d\frac{\partial V}{\partial x_j}(x_0)\nabla_{e_j} \varphi^{IV}(\mathcal H^{(0)})\sum_{k=1}^d\frac{\partial V}{\partial x_k}(x_0)\nabla_{e_k}.
\end{align*}
Now we consider the corresponding Schwartz kernels and set $Z=Z^\prime=0$:
\[
F_{2,1}(0,0) =-\frac{1}{6}\sum_{j,k=1}^d\frac{\partial V}{\partial x_j}(x_0) \frac{\partial V}{\partial x_k}(x_0)  K_{\nabla_{e_j}\nabla_{e_k}\varphi^{IV}(\mathcal H^{(0)})}(0,0).
\]
By \eqref{e:nablaPhi-kernel}, we have 
\[
K_{\nabla_{e_j}\nabla_{e_k} \varphi (\mathcal H^{(0)})}(0,0)=\frac{1}{(2\pi)^{d}} \int_{\RR^{d}} (i\xi_j)(i\xi_k) \varphi(|\xi|^2+V(x_0))d\xi.
\]
Therefore, if $j\neq k$, we have 
\begin{equation}\label{e:nablaejkPhi-kernel}
K_{\nabla_{e_j}\nabla_{e_k} \varphi (\mathcal H^{(0)})}(0,0)=0,
\end{equation}
and, if $j=k$, integrating by parts, we get 
\begin{multline}\label{e:nablaejjPhi-kernel}
K_{\nabla_{e_j}^2 \psi^\prime (\mathcal H^{(0)})}(0,0)=-\frac{1}{(2\pi)^{d}} \int_{\RR^{d}} \xi_j^2 \psi^\prime(|\xi|^2+V(x_0))d\xi\\ =\frac 12 \frac{1}{(2\pi)^{d}} \int_{\RR^{d}}\psi(|\xi|^2+V(x_0))d\xi.
\end{multline}
We conclude that
\begin{equation}
\label{e:F2100}
F_{2,1}(0,0) =-\frac{1}{12}\sum_{j=1}^d\left(\frac{\partial V}{\partial x_j}(x_0)\right)^2 \frac{1}{(2\pi)^{d}} \int_{\RR^{d}}\varphi^{\prime\prime\prime}(|\xi|^2+V(x_0))d\xi.
\end{equation}


If $g$ is a polynomial of degree $\leq 2$, then, by iterating \eqref{e:comm-g-res}, we get
\begin{equation}
\begin{aligned}\label{e:comm-g-res2}
(\lambda-\mathcal H^{(0)})^{-1}g= &g(\lambda-\mathcal H^{(0)})^{-1}\\
&+\left(-2\sum_{k=1}^d  \frac{\partial g}{\partial Z_k}\nabla_{e_k}+ \Delta g\right)(\lambda-\mathcal H^{(0)})^{-2}\\
&+4\sum_{j,k=1}^d\frac{\partial^2 g}{\partial Z_j\partial Z_k}\nabla_{e_j}\nabla_{e_k} (\lambda-\mathcal H^{(0)})^{-3}.
\end{aligned}
\end{equation}

%
%
%

By \eqref{e:comm-g-res2}, we get
\begin{multline}
(\lambda-\mathcal H^{(0)})^{-1}\mathcal H^{(2)}(\lambda-\mathcal H^{(0)})^{-1}\\
\begin{aligned}\label{e:H0H2H0}
= &\mathcal H^{(2)}(\lambda-\mathcal H^{(0)})^{-2}\\
&+\left(-2\sum_{k=1}^d \frac{\partial \mathcal H^{(2)}}{\partial Z_k}\nabla_{e_k}+ \Delta \mathcal H^{(2)}\right)(\lambda-\mathcal H^{(0)})^{-3}\\
&+4 \sum_{j,k=1}^d\frac{\partial^2 \mathcal H^{(2)}}{\partial Z_j\partial Z_k} \nabla_{e_j}\nabla_{e_k} (\lambda-\mathcal H^{(0)})^{-4},
\end{aligned}
\end{multline}
where $\mathcal H^{(2)}$ is given by \eqref{e:H2} and $\frac{\partial \mathcal H^{(2)}}{\partial Z_k}$, $\Delta\mathcal H^{(2)}$ and $\frac{\partial^2 \mathcal H^{(2)}}{\partial Z_j\partial Z_k}$ are obtained from  $\mathcal H^{(2)}$ by applying the relevant differential operator to its coefficients: 
\begin{align*}
\frac{\partial \mathcal H^{(2)}}{\partial Z_k}=& \frac{2}{3}i\sum_{j,\ell=1}^d\frac{\partial B_{jk}}{\partial x_\ell}(x_0)  \,Z_\ell \nabla_{e_j}+\frac{2}{3}i\sum_{j,\ell=1}^d \frac{\partial B_{j\ell}}{\partial x_k}(x_0)  \,Z_\ell \nabla_{e_j}\\ &+\frac{1}{3}\sum_{j=1}^d\frac{\partial B_{jk}}{\partial x_j}(x_0)  +\frac{1}{2}\sum_{j,\ell=1}^dB_{k j}(x_0)B_{\ell j}(x_0)\,Z_\ell + \sum_{j=1}^d\frac{\partial^2 V}{\partial x_j\partial x_k}(x_0)Z_j,\\
\Delta\mathcal H^{(2)}=& -\frac{4}{3}i\sum_{j,k=1}^d\frac{\partial B_{jk}}{\partial x_k}(x_0) \, \nabla_{e_j}-\frac{1}{2}\sum_{j,k=1}^dB_{k j}(x_0)^2 -\sum_{j=1}^d\frac{\partial^2 V}{\partial x^2_j}(x_0).\\
\frac{\partial^2 \mathcal H^{(2)}}{\partial Z_j\partial Z_k}=&\frac{2}{3}i\sum_{m=1}^d\left( \frac{\partial B_{mj}}{\partial x_k}(x_0)  \nabla_{e_m}+\frac{\partial B_{mk}}{\partial x_j}(x_0)\nabla_{e_m}\right)\\ &+\frac{1}{2}\sum_{m=1}^dB_{j m}(x_0)B_{k m}(x_0) +\frac{\partial^2 V}{\partial x_j\partial x_k}(x_0).
\end{align*}
By \eqref{e:F22}, \eqref{e:H0H2H0} and \eqref{e:phik}, we get
\begin{multline}
F_{2,2}= \mathcal H^{(2)}\varphi^{\prime}(H^{(0)})
+\frac 12 \left(-2\sum_{k=1}^d \frac{\partial \mathcal H^{(2)}}{\partial Z_k}\nabla_{e_k}+ \Delta \mathcal H^{(2)}\right)\varphi^{\prime\prime}(H^{(0)})\\
+\frac 23\sum_{j,k=1}^d\frac{\partial^2 \mathcal H^{(2)}}{\partial Z_j\partial Z_k}\nabla_{e_j}\nabla_{e_k} \varphi^{\prime\prime\prime}(H^{(0)}).
\end{multline}
Now we consider the corresponding Schwartz kernels and set $Z=Z^\prime=0$:
\begin{multline}\label{e:F2200}
F_{2,2}(0,0)=-\frac 12 \left(\frac{1}{2}\sum_{j,k=1}^dB_{k j}(x_0)^2 +\sum_{j=1}^d\frac{\partial^2 V}{\partial x^2_j}(x_0)\right)K_{\varphi^{\prime\prime}(H^{(0)})}(0,0)\\
+\frac 13  \sum_{j,k=1}^d\left(\frac{1}{2}\sum_{m=1}^dB_{j m}(x_0)B_{k m}(x_0) +\frac{\partial^2 V}{\partial x_j\partial x_k}(x_0)\right) K_{\nabla_{e_j}\nabla_{e_k}\varphi^{\prime\prime\prime}(H^{(0)})}(0,0)\\
=-\frac{1}{6} \left(\frac{1}{2}\sum_{j,k=1}^dB_{k j}(x_0)^2 +\sum_{j=1}^d\frac{\partial^2 V}{\partial x^2_j}(x_0)\right)K_{\varphi^{\prime\prime}(H^{(0)})}(0,0).
\end{multline}
By \eqref{e:F2100} and \eqref{e:F2200}, we get the formula \eqref{e:exf2}.

\end{document}